\documentclass[addpoints,11pt]{article}

\usepackage{amsfonts}
\usepackage{amsmath}
\usepackage{amssymb}
\usepackage{amsthm}
\usepackage{authblk}
\usepackage{array}
\usepackage{chngpage}
\usepackage{color}
\usepackage{commath}
\usepackage{comment}
\usepackage[shortlabels]{enumitem}
\usepackage{fge}
\usepackage{graphics}
\usepackage{graphicx}
\usepackage[hidelinks]{hyperref}
\usepackage{latexsym}
\usepackage{mathrsfs}
\usepackage{mathtools}
\usepackage{multicol}
\usepackage{multirow}
\usepackage{nccmath}
\usepackage{pgfplots}
\usepackage{textcomp}
\usepackage{tikz}
\usepackage[utf8x]{inputenc}

\usepackage{lmodern}
\usepackage[T1]{fontenc}

\newcommand{\PreserveBackslash}[1]{\let\temp=\\#1\let\\=\temp}
\newcolumntype{C}[1]{>{\PreserveBackslash\centering}p{#1}}
\newcolumntype{R}[1]{>{\PreserveBackslash\raggedleft}p{#1}}
\newcolumntype{L}[1]{>{\PreserveBackslash\raggedright}p{#1}}

\DeclareSymbolFont{eulargesymbols}{U}{zeuex}{m}{n}
\DeclareMathSymbol{\intop}{\mathop}{eulargesymbols}{"52}
\DeclareMathSymbol{\ointop}{\mathop}{eulargesymbols}{"49}

\newcommand{\al}{\alpha}
\newcommand{\bC}{\mathbb C}

\newcommand{\be}{\beta}

\newcommand{\bR}{\mathbb R}

\newcommand{\bT}{\mathbb T}
\newcommand{\bZ}{\mathbb Z}

\newcommand{\cI}{\mathcal I}

\newcommand{\de}{\delta}

\newcommand{\e}{\mathrm{e}}

\newcommand{\eps}{\varepsilon}
\newcommand{\ga}{\gamma}

\newcommand{\la}{\lambda}

\newcommand{\nf}{\infty}

\newcommand{\Ph}{\Phi}
\newcommand{\si}{\sigma}

\newcommand{\tht}{\theta}

\newcommand{\x}{\raisebox{0.3mm}{\mbox{\tiny $\times$}\hspace{-0.3mm}}}

\renewcommand{\ge}{\geqslant}

\renewcommand{\i}{\mathrm{i}}

\renewcommand{\le}{\leqslant}

\def\XXint#1#2#3{{\setbox0 = \hbox{$#1{#2#3}{ \int}$ }
\vcenter{\hbox{$#2#3$ }}\kern-.57\wd0}}

\usetikzlibrary{arrows,shapes,positioning}
\usetikzlibrary{decorations.markings}
\tikzstyle arrowstyle=[scale=1]
\tikzstyle directed=[postaction={decorate,decoration={markings,
    mark=at position .65 with {\arrow[arrowstyle]{stealth}}}}]
\tikzstyle reverse directed=[postaction={decorate,decoration={markings,
    mark=at position .65 with {\arrowreversed[arrowstyle]{stealth};}}}]

\numberwithin{equation}{section}

\newtheorem{lemma}{Lemma}[section]
\newtheorem{theorem}[lemma]{Theorem}

\newtheorem{remark}[lemma]{Remark}

\begin{document}

\title{\huge\bf Fine spectral estimates with applications to the optimally fast solution of large FDE linear systems}

\author[1]{M.~Bogoya}
\author[2]{S.M.~Grudsky}
\author[3]{S.~Serra--Capizzano}
\author[4]{C.~Tablino--Possio}

\affil[1]{\scriptsize University of Insubria, Como, Italy\\
Dipartimento di Scienza e Alta Tecnologia\\
johanmanuel.bogoya@uninsubria.it}
\affil[2]{\scriptsize CINVESTAV del I.P.N., Mexico D.F., Mexico\\
Departamento de Matem\'aticas\\
Southern Federal University, Rostov-on-Don, Russia\\
Regional Mathematical Center\\
grudsky@math.cinvestav.mx}
\affil[3]{\scriptsize University of Insubria, Como, Italy\\
Dipartimento di Scienza e Alta Tecnologia\\ s.serracapizzano@uninsubria.it}
\affil[4]{\scriptsize University of Milano-Bicocca, Milano, Italy\\
Dipartamento di Matematica e Applicazioni\\ cristina.tablinopossio@unimib.it}

\maketitle

\begin{abstract}
\noindent In the present note we consider a type of matrices stemming in the context of the numerical approximation of distributed order fractional differential equations (FDEs): from one side they could look standard, since they are, real, symmetric and positive definite. On the other hand they present specific difficulties which prevent the successful use of classical tools. In particular the associated matrix-sequence, with respect to the matrix-size, is ill-conditioned and it is such that a generating function does not exists, but we face the problem of dealing with a sequence of generating functions with an intricate expression.
Nevertheless, we obtain a real interval where the smallest eigenvalue belongs, showing also its asymptotic behavior.
We observe that the new bounds improve those already present in the literature and give a more accurate spectral information, which are in fact used in the design of fast numerical algorithms for the associated large linear systems, approximating the given distributed order FDEs. Very satisfactory numerical results are presented and critically discussed, while a section with conclusions and open problems ends the current note.\\[3ex]
\textbf{Subclass:} MSC 15B05, 15A18, 26A33
\end{abstract}

\section{Introduction}

When considering the numerical approximation of fractional differential equations (FDEs), due to the nonlocal nature of the involved operators, the matrices are intrinsically dense, even in the case where the approximation technique is of local nature (Finite Differences, Finite Elements, Finite Volumes, Isogeometric Analysis etc.); see \cite{ErRo06,HeMo14,LiWa18,MaKa18,MeTa04,PeTa11,TiZh15,WaDu13,XuDa20,ZeMa17} and references therein. The same situation is present also in the context of distributed order FDEs (see \cite{MaSe21} and references therein). When employing equispaced gridding, the resulting structures have the advantage of being of Toeplitz nature, so that the cost of a matrix-vector multiplication is almost linear i.e. of $O(n\log n)$ arithmetic operations, where $n$ is the matrix order and the constant hidden in the big $O$ is very moderate (refer to \cite{ChJi07,Ng04} and to the references there reported). These computational features, joint with the usually large dimensions of the considered linear systems, lead to the search for suitable iterative solvers: typically the most successful iterative procedures belong to the class of preconditioned Krylov methods, to the algorithms of multigrid type, or to clever combinations of them \cite{ChJi07,DoGa15,DoGa17,DoMa16,Ng04,Se02a,Se99b}.

A crucial information for an appropriate design of an efficient solver of such a kind is the precise knowledge of the asymptotical behavior of the minimal eigenvalue, which in the positive definite case is also related to the asymptotic spectral conditioning: in our setting we emphasize that the considered matrices are real, symmetric, and positive definite.

In this work, starting from \cite{BoGr21,MaSe21}, we improve the bounds present in the literature. We take into account the techniques already developed in \cite{BoBo16,BoGr98}: the additional difficulty, in the present setting, relies on the fact that the generating function is not fixed, since it depends on the matrix order $n$, which is somehow a challenging novelty with respect to the classical work on the subject \cite{BoBo16,BoGr98,Se98}.

In reality, in technical terms, we consider the real-valued symbol $f_{n}(\tht)\coloneqq\sum_{j=0}^{n-1}h^{jh}|\tht|^{2-jh}$ with $h\coloneqq\frac{1}{n}$. In this note we obtain a real interval where the smallest eigenvalue of the Toeplitz matrix $A_{n}\coloneqq T_{n}(f_{n})$ belongs to, with $f_{n}$ being the generating function of $A_{n}$, for any fixed $n$. 

Based on the spectral information, few algorithmic proposals are also discussed starting from those presented in \cite{DoMa16,DoMa18,MaSe21} and are of the type mentioned before: preconditioned conjugate gradient (PCG) algorithms, multigrid solvers (typically the V-cycle), and combinations of them
that is a V-cycle where the smoothing iterations (one step) incorporate the proposed PCG choices.

The work is organized as follows. Section~\ref{sc:setting-sol} contains the setting of the problem regarding the minimal eigenvalue of $A_{n}$ and its study and solution. Section~\ref{sc:numerics} is devoted to numerical experiments concerning the solution of large linear systems with coefficient matrix
$A_{n}$, including few evidences of the spectral clustering at one of the proposed preconditioned matrix-sequences.

\section{The problem and its solution}\label{sc:setting-sol}

Let $h\coloneqq\frac{1}{n}$ and consider the function $f_{n}\colon[-\pi,\pi]\to\bR$ given by
\begin{equation*}
f_{n}(\tht)\coloneqq\tht^{2}\sum_{j=0}^{n-1} h^{jh}|\tht|^{-jh}
=\tht^{2}\frac{1-\frac{1}{n|\tht|}}{1-\big(\frac{1}{n|\tht|}\big)^{\frac{1}{n}}}.
\end{equation*}
Note that $f_{n}$ coincides with the function $\hat F_{n}$ used in \cite{BoGr21}. We employ the notation $A_{n}\coloneqq T_{n}(f_{n})$, then $f_{n}$ is the \textit{so-called} generating function or symbol of $A_{n}$. Using that,
\begin{eqnarray*}
1-\Big(\frac{1}{n|\tht|}\Big)^{\frac{1}{n}}&=&1-\exp\Big\{-\frac{\log(n|\tht|)}{n}\Big\}\\
&=&\frac{\log(n|\tht|)}{n}\Big\{1+O\Big(\frac{|\log(|n|\tht|)|}{n}\Big)\Big\}\quad\mbox{as}\quad n\to\nf,
\end{eqnarray*}
we obtain
\begin{eqnarray*}
f_{n}(\tht)&=&\frac{n\tht^{2}}{\log(n|\tht|)}\Big\{1-\frac{1}{n|\tht|}\Big\}\Big\{1+O\Big(\frac{|\log(n|\tht|)|}{n}\Big)\Big\}\\
&=&\frac{|\tht|(n|\tht|-1)}{|\log(n|\tht|)|}\Big\{1+O\Big(\frac{|\log(n|\tht|)|}{n}\Big)\Big\}\\
&=&\frac{|\tht|(n|\tht|-1)}{\log(n|\tht|)}+O\Big(\tht^{2}+\frac{|\tht|}{n}\Big).
\end{eqnarray*}
To simplify the previous expression, we consider now the function $g$ given by
\begin{equation}\label{eq:f}
g(\si)\coloneqq\frac{\si^{2}-\si}{\log(\si)}.
\end{equation}
Then we infer the relation $nf_{n}(\tht)=g(n|\tht|)+r_{n}(\tht)$ where $r_{n}(\tht)=O(n\tht^{2}+|\tht|)$ as $n\to\nf$.

We now establish a connection with the simple-loop method \cite{BoBo15a,BoBo16}. Consider the Toeplitz matrix $T_{n}(a)$ with symbol $a\colon\bT\to\bR$ given by
\[a(t)=-\frac{1}{t}+2-t=4\sin^{2}\Big(\frac{\tht}{2}\Big),\]
where $t=\e^{\i\tht}$. For a symbol $b$, let $\la_{j}(T_{n}(b))$ and $\psi_{j}(T_{n}(b))$ be its $j$th eigenvalue and eigenfunction, respectively. Since $a$ is real-valued, its eigenvalues must be real and we arrange them in increasing order, that is
\[\la_{1}(T_{n}(a))\le\la_{2}(T_{n}(a))\le\cdots\le\la_{n}(T_{n}(a)).\]
If $\psi_{j}(T_{n}(a))$ is the eigenfunction corresponding to $\la_{j}(T_{n}(a))$, then it is well-known that (see for example \cite{BoBo16,BoGr05}) it leads to the equations
\begin{eqnarray}\label{eq:lapsi}
\la_{1}(T_{n}(a))&=&4\sin^{2}\Big(\frac{s}{2}\Big),\notag\\
\psi_{1}(T_{n}(a))(\tht)&=&\frac{c_{n}\, \e^{\frac{\i}{2}(n+1)\tht}}{(n+1)^{\frac{3}{2}}}\cdot\frac{\cos\big(\frac{(n+1)\tht}{2}\big)}{\sin\big(\frac{\tht-s}{2}\big)\sin\big(\frac{\tht+s}{2}\big)},
\end{eqnarray}
where $s\coloneqq\frac{\pi}{n+1}$ and $c_{n}$ is bounded when $n\to\nf$. The constant $c_{n}$ can be calculated from the relation $\|\psi_{1}(T_{n}(a))\|_{2}=1$, see Figure~\ref{fg:cn}.

\begin{figure}[ht]
\centering
\includegraphics[width=9cm]{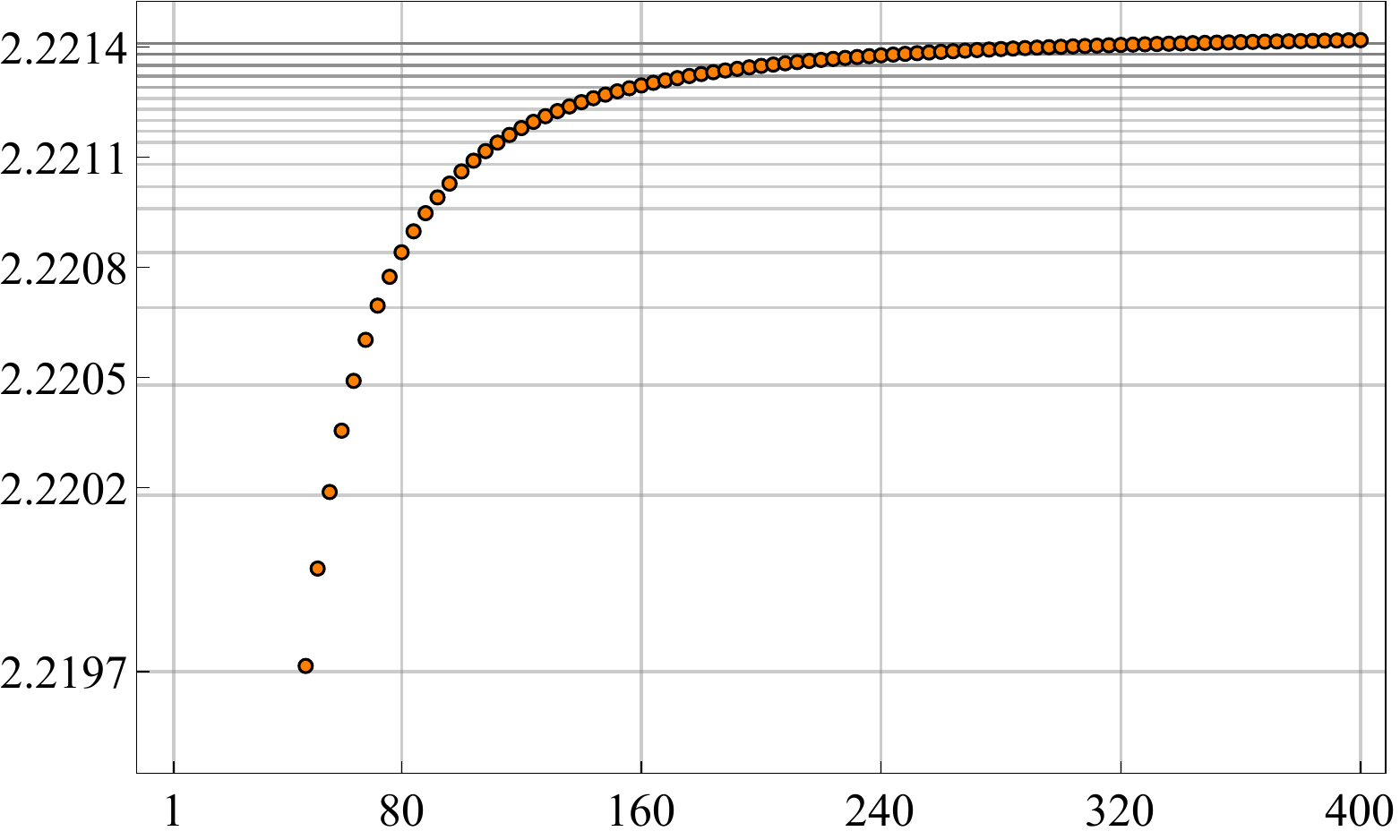}
\caption{The constant $c_{n}$ in Equation \eqref{eq:lapsi} for different values of $n$.}\label{fg:cn}
\end{figure}
\bigskip

\subsection{Upper bound for $\mathbf{\la_{1}(A_{n})}$}

\begin{theorem}\label{th:UBound}
We have
\[n\la_{1}(A_{n})\le k_{1}+O\Big(\frac{1}{n}\Big),\quad\mbox{as}\quad n\to\nf,\]
where the constant $k_{1}$ is given by
\[k_{1}=\Big\{\int_{0}^{\nf} \frac{u^{2}-u}{(u^{2}-\pi^{2})^{2}}\cdot\frac{\cos^{2}(\frac{u}{2})}{\log(u)}\dif u\Big\}\Big\{\int_{0}^{\pi}\frac{\cos^{2}(\frac{u}{2})}{(u^{2}-\pi^{2})^{2}}\dif u\Big\}^{-1},\]
and can be numerically approximated as $12.9301$.
\end{theorem}

\begin{proof}
Let $\langle\cdot,\cdot\rangle$ be the inner product of the Hilbert space $L^{2}(\bT)$. Note that the essential ranges of the symbols $n^{2} a$ and $n f_{n}$ have a similar behavior in an interval of the type $\big[0,O\big(\frac{1}{n}\big)\big]$.
Hence using the well-known formula
\begin{equation*}
\la_{1}(A_{n})=\inf\{\langle A_{n}\psi,\psi\rangle\colon\|\psi\|_{2}=1\},
\end{equation*}
we obtain
\begin{eqnarray}\label{eq:laupL}
\la_{1}(A_{n})&\le&\langle A_{n}\psi_{1}(A_{n}),\psi_{1}(A_{n})\rangle\notag\\
&=&\frac{1}{2\pi}\int_{-\pi}^{\pi} f_{n}(\tht)|\psi_{1}(A_{n})(\tht)|^{2}\dif\tht\notag\\
&=&\cI_{1}(n)+\cI_{2}(n),
\end{eqnarray}
where
\begin{eqnarray*}
\cI_{1}(n)&\coloneqq&\frac{1}{2\pi n}\int_{-\pi}^{\pi} g(n|\tht|)|\psi_{1}(A_{n})(\tht)|^{2}\dif\tht,\\
\cI_{2}(n)&\coloneqq&\frac{1}{2\pi n}\int_{-\pi}^{\pi}r_{n}(\tht)|\psi_{1}(A_{n})(\tht)|^{2}\dif\tht.
\end{eqnarray*}
It is easy to see that $\lim_{n\to\nf}n\,\cI_{1}(n)<\nf$, indeed
\begin{eqnarray}\label{eq:laupL1}
n\,\cI_{1}(n)&=&\frac{c_{n}^{2}}{2\pi(n+1)^{3}}\int_{-\pi}^{\pi} g(n|\tht|)\cdot\frac{\cos^{2}\big(\frac{(n+1)\tht}{2}\big)}{\sin^{2}\big(\frac{\tht-s}{2}\big)\sin^{2}\big(\frac{\tht+s}{2}\big)}\dif\tht\notag\\
&\sim&\frac{16nc_{n}^{2}}{\pi}\int_{0}^{\pi} g(n\tht)\cdot\frac{\cos^{2}\big(\frac{n\tht}{2}\big)}{((n\tht)^{2}-\pi^{2})^{2}}\dif\tht\notag\\
&\sim&\frac{16c_{n}^{2}}{\pi}\int_{0}^{\nf} \frac{u^{2}-u}{(u^{2}-\pi^{2})^{2}}\cdot\frac{\cos^{2}\big(\frac{u}{2}\big)}{\log(u)}\dif u.
\end{eqnarray}
Here the notation $f\sim g$ means $\lim_{n\to\nf}\frac{f(n)}{g(n)}=1$. Using \eqref{eq:lapsi} a similar calculation produces
\begin{eqnarray}\label{eq:cnotn}
c_{n}^{-2}&\sim&\frac{16n}{\pi}\int_{0}^{\pi}\frac{\cos^{2}\big(\frac{n\tht}{2}\big)}{((n\tht)^{2}-\pi^{2})^{2}}\dif\tht\notag\\
&=&\frac{16}{\pi}\int_{0}^{\nf}\frac{\cos^{2}\big(\frac{u}{2}\big)}{(u^{2}-\pi^{2})^{2}}\dif u.
\end{eqnarray}
Note that the last integral does not depend on $n$, its value can be numerically found producing $c_{n}\approx2.2214$, which agrees with the Figure \ref{fg:cn}. For the second integral in \eqref{eq:laupL} we recall that $r_{n}(\tht)=O(n\tht^{2}+|\tht|)$ and write $\cI_{2}(n)=O(\Ph(n))$ where
\begin{eqnarray*}
\Ph(n)&=&\frac{1}{n^{2}}\int_{-\pi}^{\pi} \hspace{-2mm}\{(n|\tht|)^{2}-n|\tht|\}|\psi_{1}(A_{n})(\tht)|^{2}\dif\tht\\
&=&\frac{c_{n}^{2}}{n^5}\int_{-\pi}^{\pi} \hspace{-2mm}\{(n|\tht|)^{2}-n|\tht|\}\frac{\cos^{2}\big(\frac{n|\tht|}{2}\big)}{((n|\tht|)^{2}-\pi^{2})^{2}}\dif\tht\\
&\sim&\frac{16c_{n}^{2}}{n^{2}}\int_{0}^{\nf} \frac{u^{2}-u}{(u^{2}-\pi^{2})^{2}}\cos\Big(\frac{u}{2}\Big)\dif u,
\end{eqnarray*}
which combined with \eqref{eq:cnotn} produces
\begin{equation}\label{eq:laupL2}
\cI_{2}(n)=O\Big(\frac{1}{n^{2}}\Big).
\end{equation}
Finally, combining \eqref{eq:laupL}, \eqref{eq:laupL1}, \eqref{eq:cnotn}, and \eqref{eq:laupL2} we obtain the thesis.
\end{proof}

\subsection{Lower bound for $\mathbf{\la_{1}(A_{n})}$}

In this part we will implement the trick used in \cite{BoGr98} which works as follows. Let $b\in C(\bT)$ and $q_{n}\colon\bT\to\bC$ given by
\begin{equation}\label{eq:hn}
q_{n}(t)\coloneqq\sum_{j=n}^{\nf} q_{j,n} t^j+\sum_{j=n}^{\nf} q_{-j,n}t^{-j}.
\end{equation}
Since $T_{n}(q_{n})$ is the zero matrix we clearly have $T_{n}(a)=T_{n}(a+q_{n})$, thus instead of working with $T_{n}(a)$ we can use $T_{n}(a+q_{n})$ which under a swiftly selected symbol $q_{n}$, can be advantageous.

\begin{theorem}\label{th:LBound}
We have
\[n\la_{1}(A_{n})\ge k_{2}+O\Big(\frac{1}{n}\Big),\quad\mbox{as}\quad n\to\nf,\]
where the constant $k_{2}$ is given by
\[k_{2}=\frac{1}{\pi}\int_{0}^{\pi}\frac{\si^{2}-\si}{\log(\si)}\dif\si,\]
which can be approximated numerically as $2.2945$.
\end{theorem}

\begin{proof}
In our case instead of working with the symbol $nf_{n}$ we will use
\begin{eqnarray*}
g_{n}(\tht)&\coloneqq&nf_{n}(\tht)+p(n|\tht|),\\
&=&g(n|\tht|)+p(n|\tht|)+r_{n}(\tht),
\end{eqnarray*}
where $p(\si)\coloneqq\sum_{j=1}^{\nf} p_j\cos(\si j)$, $p_j$ are real constants, and $r_{n}$ is given in \eqref{eq:f}. Let $t=\e^{\i\tht}$, then we can write
\[p(n|\tht|)=p(n\tht)=\sum_{j=1}^{\nf} \frac{p_j}{2} (t^{nj}+t^{-nj}),\]
which clearly shows the form \eqref{eq:hn}. Additionally, the function $p$ is $2\pi$-periodic, even, and satisfies $\int_{0}^{\pi}p(\si)\dif\si=0$.
We thus deduce $T_{n}(p(n|\tht|))=0$, and hence
\begin{eqnarray*}
T_{n}(nf_{n})&=&T_{n}(g_{n})\\
&=&T_{n}(g(n|\tht|))+T_{n}(p(n|\tht|))+T_{n}(r_{n})\\
&=&T_{n}(g(n|\tht|))+T_{n}(r_{n}).
\end{eqnarray*}
Keeping in mind that, for any real-valued symbol $\be$, the smallest eigenvalue of $T_{n}(\be)$ must be greater or equal to the infimum of $\be$, we obtain
\begin{eqnarray*}
n\la_{1}(A_{n})&\ge&\inf\{g_{n}(\tht)\colon \tht\in[-\pi,\pi]\}\\
&=&\inf\{g(n|\tht|)+p(n|\tht|)+r_{n}(\tht)\colon \tht\in[-\pi,\pi]\}\\
&=&\inf\Big\{g(\si)+p(\si)+O\Big(\frac{\si^{2}+\si}{n}\Big)\colon \si\in[0,n\pi]\Big\},
\end{eqnarray*}
where in the last line we used the variable change $\si=n|\tht|$. 
Then we obtain
\[n\la_{1}(A_{n})\ge\inf\Big\{g(\si)+p(\si)+O\Big(\frac{\si^{2}+\si}{n}\Big)\colon \si\in[0,M]\Big\},\]
for a sufficiently large constant $M$. Thus we have
\begin{equation}\label{eq:lb}
n\la_{1}(A_{n})\ge\inf\{g(\si)+p(\si)\colon \si\in[0,M]\}+O\Big(\frac{1}{n}\Big).
\end{equation}
In order to obtain a neat lower bound for $\la_{1}(A_{n})$, we need to choose the coefficients $p_j$ in such a way that
\[m\coloneqq \inf_{0\le\si\le M}\{g(\si)+p(\si)\}\]
be maximal. Since $p$ is $2\pi$-periodic and $g$ is strictly increasing with $g(0)=0$ and $g(\nf)=\nf$, the infimum of $g+p$ must be in the interval $[0,\pi]$, and consequently we infer that
\[m=\inf_{0\le\si\le\pi}\{g(\si)+p(\si)\}.\]
Consider the integral
\[k_{2}\coloneqq \frac{1}{\pi}\int_{0}^{\pi} \{g(\si)+p(\si)\}\dif\si=\frac{1}{\pi}\int_{0}^{\pi} g(\si)\dif\si,\]
which satisfies $k_{2}\ge m$, and take $p$ as
\[p(\si)=\begin{cases}
k_{2}-g(\si),&\si\in[-\pi,\pi];\\
k_{2}-g(\si-2\pi j),&\si\in[\pi j,\pi(j+2)],\quad j\in\bZ.
\end{cases}\]
It is immediate that the function $p$ is $2\pi$-periodic, $[g(\si)+p(\si)]_{\si\in[-\pi,\pi]}\equiv k_{2}$, $g(\si)+p(\si)\ge k_{2}$ for $\si\notin[-\pi,\pi]$, which implies $m=k_{2}$.
which combined with \eqref{eq:lb} proves the theorem.
\end{proof}

Finally, combining the Theorems \ref{th:UBound} and \ref{th:LBound} we obtain
\[2.2945\approx k_{2}\le n\la_{1}(A_{n}) \le k_{1}\approx 12.9301,\]
for every sufficiently large $n$.

\begin{remark}
In this remark we give a specific account on a comparison of the obtained results with respect to those available in the literature and which amount so far to only two works \cite{BoGr21,MaSe21}: our bounds are more precise as emphasized in the next section.
\end{remark}

\section{Few selected numerical experiments}\label{sc:numerics}

The current section is divided into two parts. In the first we discuss our theoretical results, regarding the bounds on the minimal eigenvalue of $A_{n}$, as a function of the matrix order $n$, also in comparison with the bounds present in the literature.
Regarding \cite{MaSe21}, based on \cite{Se95}, the only relevant observation is that the minimal eigenvalue of $T_{n}^{-1}(|\tht|^{2})h^\al T_{n}(|\tht|^{2-\al})$ is well separated from zero, for any choice of $\al \in (0,2)$, and this provides a qualitative indication that the minimal eigenvalue of ${A_{n} \over n}$ converges to zero with a speed of ${1 \over n^{2}}$.
Concerning \cite{BoGr21}, the latter claim is indeed proved formally, but the constants are not computed, while in the present note we  improve the findings, by determining quite precise lowerbounds and upperbounds
(see Figure \ref{fg:bounds}).

\begin{figure}[ht]
\centering
\includegraphics[width=75mm]{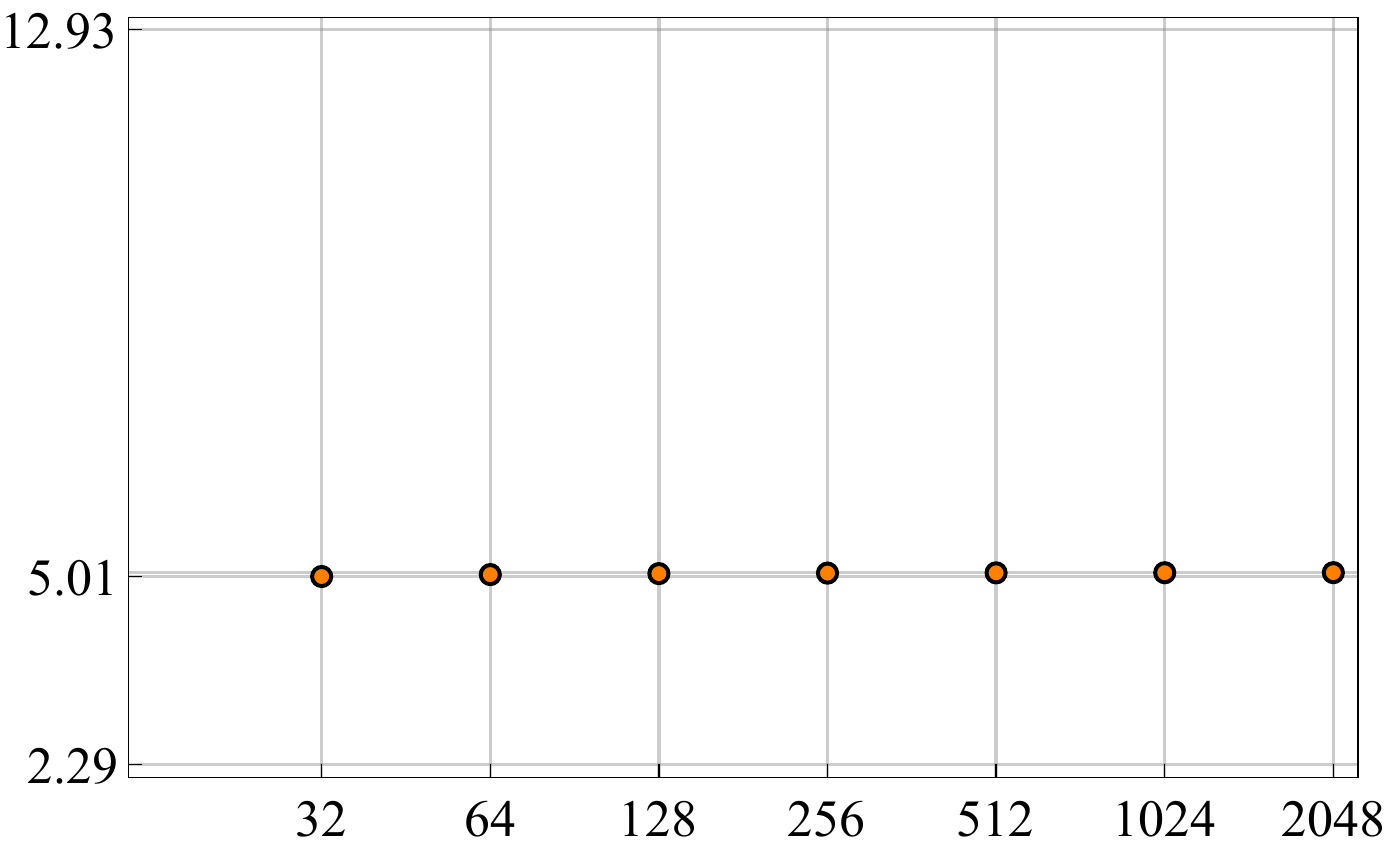}\hspace{3mm}
\includegraphics[width=75mm]{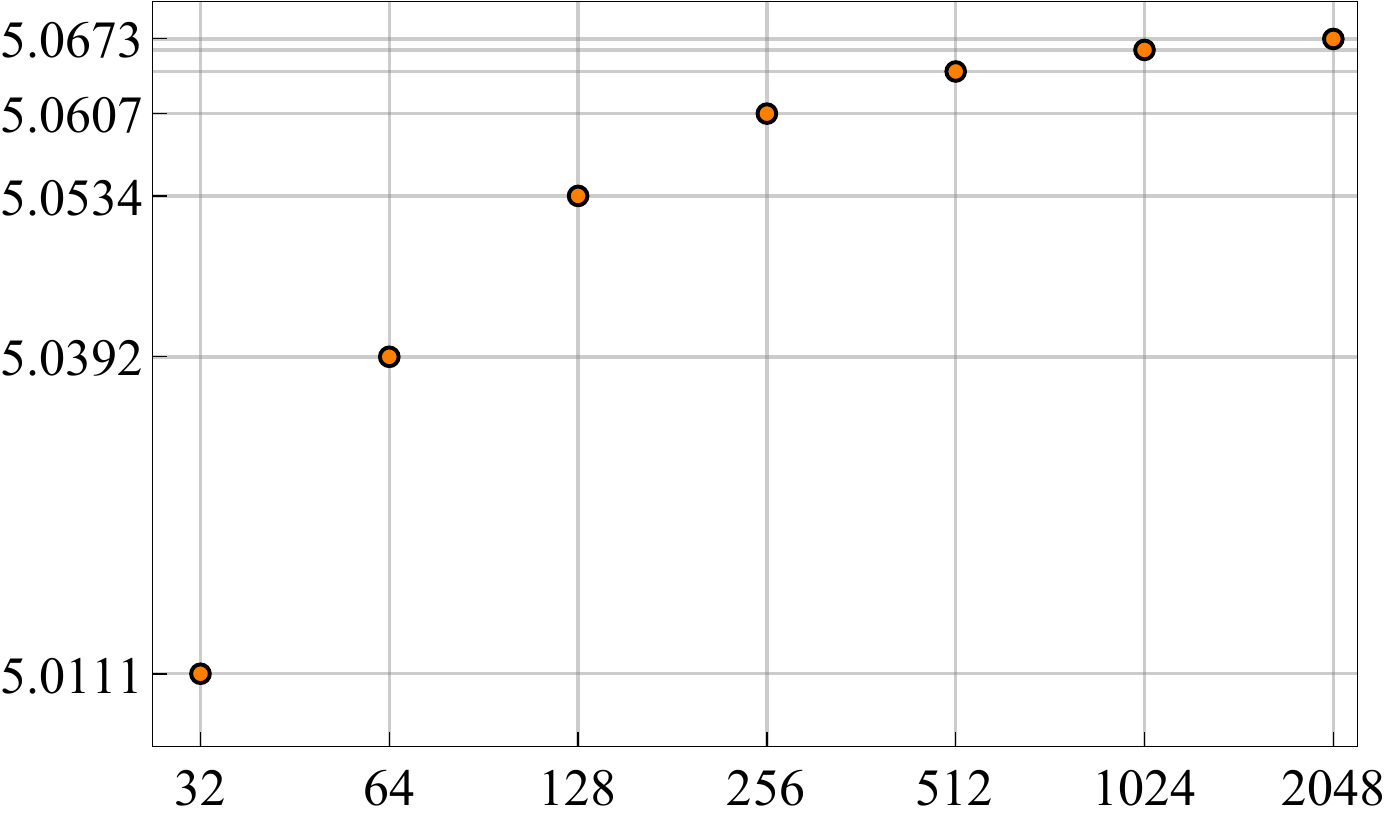}
\caption{The normalized minimum eigenvalue $n\la_{1}(A_{n})$ of $A_{n}$ for different values of $n$. The left image includes the lower and upper bounds given by Theorems \ref{th:LBound} and \ref{th:UBound}, respectively, while the right image shows a standard data range.}\label{fg:bounds}
\end{figure}

The second part is devoted to exploit the spectral information for the use and the design of specialized iterative algorithms, when solving large linear systems having ${A_{n} \over n}$ as coefficient matrix. In Table~\ref{tb:PCG} we employ the standard conjugate gradient (CG), since the coefficient matrix is positive definite, but we do not expect optimality due to the ill-conditioning of the related matrix-sequence. We then use the preconditioned conjugate gradient (PCG) with Strang and Frobenius optimal preconditioners (see \cite{ChJi07,DiBSe99,Ng04,Se99a} taking the preconditioner in the algebra of the circulant matrices and in the algebra of sine transforms, containing the Finite Difference discrete Laplacian $T_{n}(2-2\cos\tht)$).
Also in this case it is nontrivial to obtain optimality, due to the ill-conditioning of the involved matrices, which grows to infinity as the matrix size tends to infinity.

Few observations are in order: the considered matrix $A_{n}$ is real, symmetric, and Toeplitz and the $\tau$ algebra is again inherently real, symmetric, and with Toeplitz generator \cite{DiBSe99}. The latter statement represents a qualitative explanation of the fact that the $\tau$ preconditioners perform substantially better than the analogs in the circulant algebra (see Table \ref{tb:PCG}): the theoretical ground to the preceding qualitative remark relies in the notion of good algebras developed in \cite{Se99c}.

Notice that the tridiagonal $\tau$ discrete Finite Difference Laplacian $\Delta_{n}$ is optimal: the related linear system solution is extremely cheap both in a sequential model (via the Thomas algorithm) and in a parallel model of computation (via e.g. classical Multigrid methods \cite{HaWo85}). We also notice the remarkable difference between the performance of the optimal $\tau$ preconditioner and of the optimal circulant preconditioner.
The reason is spectral (plus the good algebra argument \cite{Se99c}). The minimal eigenvalue of the optimal circulant preconditioner is averaged, due to a Ces\`aro sum effect \cite{DiBSe99,Se99a}, and hence it behaves as   
${1\over n}$ (instead of $1\over n^{2}$) and this explains the reason why the number of iterations grows as $\sqrt n$, in the light of the classical results on conjugate gradient methods \cite{AxLi86}.  On the other hand, the optimal $\tau$ preconditioner matches very well the extremal eigenvalues of $A_{n}$ (see e.g. \cite{DiBSe99}). 

Thus, as a conclusion, we deduce that the preconditioners $C_{S,n}$ (Strang-Circulant), $\tau_{N,n}$ (Natural-$\tau$), $\tau_{F,n}$ (Frobenius-Optimal $\tau$), and $\Delta_{n}$ (Discrete FD Laplacian) are all optimal in the sense that the iteration count is bounded by constants independent of the matrix size $n$ and the cost per iteration is that of the Fast Fourier Transform, which amounts to $O(n\log n)$ arithmetic operations (for a formal notion of optimality see \cite{AxNe94}). 
As the data indicate, the preconditioners $\tau_{N,n}$, $\tau_{F,n}$ are the best, but also $\Delta_{n}$ is of interest given its sparsity.

In Table~\ref{tb:eig_mMC}-Table~\ref{tb:eig_mML} the minimal and maximal eigenvalues of the considered preconditioned matrices are reported to highlight the efficiency of the proposed preconditioner. For sure, both the data regarding the preconditioners $\tau_{N,n}$ and $\tau_{F,n}$ deserve further attention. The outliers analysis reported in Tables~\ref{tb:out1} and~\ref{tb:out2}, respectively, seems to show a strong cluster at $1$. For sure, a weak clustering is observed, being the outliers percentage decreasing as long as $n$ increases: we notice that the weak clustering can be deduced theoretically using the GLT theory \cite{GaSe17}, while the strong clustering is nontrivial given the ill-conditioning of $A_{n}$.

The last choice is a multigrid method designed with the use of the spectral information available. In fact, the projector and the restriction operators are those classically used when dealing with the standard discrete Laplacian $\Delta_{n}=T_{n}(2-2\cos\tht)$. Indeed, even if $A_{n} \over n$ is dense and seemingly there are no similarities with $\Delta_{n}$, from a spectral point of view both matrix-sequences have minimal eigenvalue collapsing to zero as $1\over n^{2}$, up to some moderate constants.

More precisely, the transfer operators are designed, in a very standard way, as the product of the Toeplitz matrix generated by the symbol $2+2\cos\tht$ and a proper cutting matrix (see e.g. \cite{FiSe92,HaWo85}). A pure algebraic multigrid is considered, so that matrices at coarser levels are obtained by projection via the transfer operators according to the Galerkin condition.
When setting smoothers, we tested several choices by considering a standard Gauss-Seidel iteration, or PCG with sine transform preconditioners, according to the natural approach and the Frobenius optimal choice: $\nu_{\textrm{pre}}$ steps are applied for the presmoother and $\nu_{\textrm{post}}$ steps for the postsmoother, respectively.

As it can be seen in Table~\ref{tb:MGM}, the numerical results with the multigrid choice are fully satisfactory, as well: the method is optimal \cite{AxNe94} in the sense that the iteration count is bounded by a constant independent of $n$ and the cost per iteration is proportional to that of the matrix-vector product. \\
A special mention deserves a last test in which we set as presmoother the PCG with Finite Difference discrete Laplacian preconditioner with $\nu_{\textrm{pre}}=1$, and as postsmoother the PCG with Natural-$\tau$ or Frobenius-Optimal $\tau$ preconditioner with $\nu_{\textrm{post}}=1$, but only at the finest level. In all the coarser levels the smoothers are simply chosen as one iteration of the standard Gauss-Seidel iteration, so further reducing the computational cost.
The number of required iterations is $3$ as for the case $\ga$.\\
Though the efficiency is comparable with those of the preconditioned PCG in Table~\ref{tb:PCG} in the present unilevel setting, multigrid could become the most promising choice in the multilevel one, due to the theoretical barriers studied in \cite{NoSe04,Se02b,SeTy99}, regarding the non optimality of the PCG with matrix-algebra preconditioners in the multilevel context.

\begin{table}[ht]
\centering
{\footnotesize\begin{tabular}{|r|rrrrrr|}
\hline
Size & $I_{n}$ & $C_{S,n}$ & $C_{F,n}$ & $\tau_{N,n}$ & $\tau_{F,n}$ & $\Delta_{n}$ \\ \hline
 32 & 34 & 7 & 11 & 6 & 5 & 8 \\
 64 & 73 & 8 & 14 & 5 & 5 & 9 \\
 128 & 154 & 8 & 15 & 5 & 5 & 10 \\
 256 & 307 & 8 & 18 & 5 & 5 & 10 \\
 512 & 593 & 8 & 22 & 5 & 5 & 11 \\
 1024 & 1095 & 8 & 27 & 5 & 5 & 11 \\
 2048 & 2112 & 8 & 34 & 5 & 5 & 12 \\
\hline
\end{tabular}}
\caption{Number of PCG's iterations to reach convergence with respect to scaled residual less than
$10^{-7}$ with preconditioner $I_{n}$ (no preconditioning), $C_{S,n}$ (Strang-Circulant), $C_{F,n}$ (Frobenius-Optimal Circulant), $\tau_{N,n}$ (Natural-$\tau$), $\tau_{F,n}$ (Frobenius-Optimal $\tau$), and
$\Delta_{n}$ (Discrete FD Laplacian).}\label{tb:PCG}
\end{table}

\begin{table}[ht]
\centering
{\footnotesize\begin{tabular}{|R{6mm}|R{24mm}R{24mm}|R{24mm}R{24mm}|}
\hline
Size &  $\la_{\min}(C_{S,n}^{-1}A_{n})$ & $\la_{\max}(C_{S,n}^{-1}A_{n})$
 & $\la_{\min}(C_{F,n}^{-1}A_{n})$ & $\la_{\max}(C_{F,n}^{-1}A_{n})$ \\
 \hline
$32$ & $5.2258\x10^{-1}$ & $3.4325\x10^{1}$ & $1.7524\x10^{-1}$ & $4.1234\x10^{0}$ \\
$64$ & $5.1302\x10^{-1}$ & $5.4268\x10^{1}$ & $1.1034\x10^{-1}$ & $5.6639\x10^{0}$ \\
$128$ & $5.0743\x10^{-1}$ & $9.1152\x10^{1}$ & $6.6754\x10^{-2}$ & $7.8412\x10^{0}$ \\
$256$ & $5.0419\x10^{-1}$ & $1.5814\x10^{2}$ & $3.9098\x10^{-2}$ & $1.0911\x10^{1}$ \\
$512$ & $5.0234\x10^{-1}$ & $2.8003\x10^{2}$ & $2.2338\x10^{-2}$ & $1.5232\x10^{1}$ \\
$1024$ & $5.0129\x10^{-1}$ & $5.0319\x10^{2}$ & $1.2526\x10^{-2}$ & $2.1315\x10^{1}$ \\
$2048$ & $5.0071\x10^{-1}$ & $9.1428\x10^{2}$ & $6.9259\x10^{-3}$ & $2.9876\x10^{1}$ \\
\hline
\end{tabular}}
\caption{Minimal and maximal eigenvalues of preconditioned matrices with preconditioner $C_{S,n}$ (Strang-Circulant) and $C_{F,n}$ (Frobenius-Optimal Circulant). }\label{tb:eig_mMC}
\end{table}

\begin{table}[ht]
\centering
{\footnotesize\begin{tabular}{|R{6mm}|R{24mm}R{24mm}|R{24mm}R{24mm}|}
\hline
Size
 & $\la_{\min}(\tau_{N,n}^{-1}A_{n})$ & $\la_{\max}(\tau_{N,n}^{-1}A_{n})$
 & $\la_{\min}(\tau_{F,n}^{-1}A_{n})$ & $\la_{\max}(\tau_{F,n}^{-1}A_{n})$ \\
\hline
$32$   &  $8.1574\x10^{-1}$ & $1.1475\x10^{0}$ & $9.3206\x10^{-1}$ & $1.1356\x10^{0}$\\
$64$   &  $8.0608\x10^{-1}$ & $1.1527\x10^{0}$ & $9.0938\x10^{-1}$ & $1.1463\x10^{0}$\\
$128$  &  $7.9895\x10^{-1}$ & $1.1584\x10^{0}$ & $8.8984\x10^{-1}$ & $1.1551\x10^{0}$ \\
$256$  &  $7.9396\x10^{-1}$ & $1.1641\x10^{0}$ & $8.7441\x10^{-1}$ & $1.1624\x10^{0}$ \\
$512$  &  $7.9063\x10^{-1}$ & $1.1693\x10^{0}$ & $8.6271\x10^{-1}$ & $1.1685\x10^{0}$ \\
$1024$ &  $7.8851\x10^{-1}$ & $1.1742\x10^{0}$ & $8.5401\x10^{-1}$ & $1.1737\x10^{0}$ \\
$2048$ &  $7.8727\x10^{-1}$ & $1.1785\x10^{0}$ & $8.4763\x10^{-1}$ & $1.1783\x10^{0}$ \\
\hline
\end{tabular}}
\caption{Minimal and maximal eigenvalues of preconditioned matrices with preconditioner  $\tau_{N,n}$ (Natural-$\tau$) and $\tau_{F,n}$ (Frobenius-Optimal $\tau$). }\label{tb:eig_mMtau}
\end{table}

\begin{table}[ht]
\centering
{\footnotesize\begin{tabular}{|R{6mm}|R{24mm}R{24mm}|}
\hline
Size & $\la_{\min}(\Delta_{n}^{-1}A_{n})$ & $\la_{\max}(\Delta_{n}^{-1}A_{n})$\\
 \hline
$32$  & $3.1941\x10^{-1}$ & $5.4802\x10^{-1}$\\
$64$  & $2.6529\x10^{-1}$ & $5.3509\x10^{-1}$\\
$128$ & $2.2651\x10^{-1}$ & $5.2979\x10^{-1}$\\
$256$ & $1.9750\x10^{-1}$ & $5.2727\x10^{-1}$\\
$512$ & $1.7504\x10^{-1}$ & $5.2604\x10^{-1}$\\
$1024$ & $1.5713\x10^{-1}$ & $5.2543\x10^{-1}$\\
$2048$ & $1.4252\x10^{-1}$ & $5.2513\x10^{-1}$\\
\hline
\end{tabular}}
\caption{Minimal and maximal eigenvalues of preconditioned matrices with preconditioner $\Delta_{n}$ (Discrete FD Laplacian). }\label{tb:eig_mML}
\end{table}

\begin{table}[ht]
\centering
{\footnotesize\begin{tabular}{|r|rrr|rrr|}
\hline
\multirow{2}{*}{Size} & \multicolumn{3}{c|}{$\eps=10^{-1}$} & \multicolumn{3}{c|}{$\eps=10^{-2}$} \\ 
 \cline{2-7} 
 & $n_{\textrm{out}}^l$ & $n_{\textrm{out}}^r$ & $\%$ & $n_{\textrm{out}}^l$ & $n_{\textrm{out}}^r$ & $\%$ \\ \hline
32 & 1 & 2 & 9.37\% & 3 & 4 & 21.8\% \\
64 & 1 & 2 & 4.69\% & 3 & 4 & 10.9\% \\
128 & 1 & 2 & 2.34\% & 4 & 4 & 6.25\% \\
256 & 2 & 2 & 1.56\% & 5 & 4 & 3.52\% \\
512 & 2 & 2 & 0.78\% & 5 & 4 & 1.76\% \\
1024 & 2 & 2 & 0.39\% & 5 & 4 & 0.88\% \\
2048 & 2 & 2 & 0.19\% & 6 & 4 & 0.49\% \\
\hline
\end{tabular}}
\caption{Number of left ($n_{\textrm{out}}^l$) and right ($n_{\textrm{out}}^r$) outliers (eigenvalues not belonging to $(1 -\eps, 1 + \eps)$ with $\eps=10^{-1}$ and $10^{-2}$ and their percentage with respect to the dimension in the case of Natural-$\tau$ preconditioner $\tau_{N,n}$.}\label{tb:out1}
\end{table}

\begin{table}[ht]
\centering
{\footnotesize\begin{tabular}{|r|rrr|rrr|}
\hline
\multirow{2}{*}{Size} & \multicolumn{3}{c|}{$\eps=10^{-1}$} & \multicolumn{3}{c|}{$\eps=10^{-2}$} \\
 \cline{2-7}
  & $n_{\textrm{out}}^l$ & $n_{\textrm{out}}^r$ & $\%$ & $n_{\textrm{out}}^l$ & $n_{\textrm{out}}^r$ & $\%$ \\
\hline
32 & 0 & 2 & 6.25\% & 18 & 4 & 68.7\% \\
64 & 0 & 2 & 3.13\% & 2 & 6 & 12.5\% \\
128 & 1 & 2 & 2.34\% & 2 & 9 & 8.59\% \\
256 & 1 & 2 & 1.17\% & 3 & 10 & 5.08\% \\
512 & 2 & 2 & 0.78\% & 4 & 10 & 2.73\% \\
1024 & 2 & 2 & 0.39\% & 4 & 10 & 1.36\% \\
2048 & 2 & 2 & 0.19\% & 4 & 10 & 0.68\% \\
\hline
\end{tabular}}
\caption{Number of left ($n_{\textrm{out}}^l$) and right ($n_{\textrm{out}}^r$) outliers (eigenvalues not belonging to $(1 -\eps, 1 + \eps)$ with $\eps=10^{-1}$ and $10^{-2}$ and their percentage with respect the dimension in the case of Frobenius-Optimal $\tau$ preconditioner $\tau_{F,n}$.}\label{tb:out2}
\end{table}

\begin{table}[ht!]
\centering
{\footnotesize\begin{tabular}{|r|rr|rr|rr|rr|}
\hline
\multirow{2}{*}{Size} & \multicolumn{2}{c|}{Case $\al$} & \multicolumn{2}{c|}{Case $\be$} & \multicolumn{2}{c|}{Case $\ga$} & \multicolumn{2}{c|}{Case $\de$} \\
\cline{2-9}
 & TGM & Vcycle & TGM & Vcycle & TGM & Vcycle & TGM & Vcycle \\
\hline
 31   & 9  & 9  & 4 & 4 & 3 & 3 & 2 & 2 \\
 63   & 10 & 10 & 4 & 4 & 3 & 3 & 2 & 2 \\
 127  & 11 & 11 & 4 & 4 & 3 & 3 & 2 & 2 \\
 255  & 11 & 11 & 4 & 4 & 3 & 3 & 2 & 2 \\
 511  & 11 & 11 & 3 & 3 & 3 & 3 & 2 & 2 \\
 1023 & 11 & 11 & 4 & 4 & 3 & 3 & 2 & 2 \\
 2047 & 11 & 11 & 3 & 3 & 3 & 3 & 2 & 2 \\
\hline
\end{tabular}}
\caption{Number of Multigrid iterations to reach convergence with respect to scaled residual less than
$10^{-7}$: Case $\al$ = Gauss-Seidel ($\nu_{\textrm{pre}}=1$)/Gauss--Seidel ($\nu_{\textrm{post}}=1$), case $\be$ = Gauss--Seidel ($\nu_{\textrm{pre}}=1$) /PCG with Natural-$\tau$ or Frobenius-Optimal $\tau$ preconditioner ($\nu_{\textrm{post}}=1$), case $\ga$ = PCG with Discrete FD Laplacian preconditioner ($\nu_{\textrm{pre}}=1$ or $2$) /PCG with
Natural-$\tau$ or Frobenius-Optimal $\tau$ preconditioner ($\nu_{\textrm{post}}=1$), case $\de$ = PCG with
Discrete FD Laplacian preconditioner ($\nu_{\textrm{pre}}=1$) /PCG with Natural-$\tau$ or Frobenius-Optimal $\tau$ preconditioner this time ($\nu_{\textrm{post}}=2$).}\label{tb:MGM}
\end{table}
\clearpage

\section{Conclusions}\label{sc:final}

In the current note we have considered a type of matrix stemming when considering the numerical approximations of distributed order FDEs (see \cite{BoGr21,MaSe21} for example). The main contribution relies in precise bounds for the minimal eigenvalue of the involved matrices. In fact the new presented bounds improve those already present in the literature and give a more accurate spectral information. The latter knowledge has been used in the design of fast numerical algorithms for the associated linear systems approximating the given distributed order FDEs: an interesting challenge is to consider a $d$-dimensional version of the considered FDE (see \cite{MaSe21}), in order to show how the presented techniques and numerical methods can be adapted and extended in $d$-level setting with $d\ge 2$.

\bibliographystyle{acm}
\bibliography{Toeplitz}

\begin{thebibliography}{10}

\bibitem{AxLi86}
{\sc Axelsson, O., and Lindskog, G.}
\newblock On the rate of convergence of the preconditioned conjugate gradient
  method.
\newblock {\em Numer. Math. 48}, 5 (1986), 499--523.

\bibitem{AxNe94}
{\sc Axelsson, O., and Neytcheva, M.}
\newblock The algebraic multilevel iteration methods-theory and applications.
\newblock In {\em Proceedings of the Second International Colloquium on
  Numerical Analysis\/} (1994), VSP, Utrecht, pp.~13--24.

\bibitem{BoBo15a}
{\sc Bogoya, M., B\"ottcher, A., Grudsky, {\relax S.M}., and Maximenko, {\relax
  E.A}.}
\newblock Eigenvalues of {H}ermitian {T}oeplitz matrices with smooth
  simple-loop symbols.
\newblock {\em Oper. Theory: Adv. and Appl. 422\/} (2015), 1308--1334.

\bibitem{BoBo16}
{\sc Bogoya, M., B\"ottcher, A., Grudsky, {\relax S.M}., and Maximenko, {\relax
  E.A}.}
\newblock Eigenvectors of {H}ermitian {T}oeplitz matrices with smooth
  simple-loop symbols.
\newblock {\em Linear Algebra Appl. 493\/} (2016), 606--637.

\bibitem{BoGr21}
{\sc Bogoya, M., Grudsky, {\relax S.M}., Mazza, M., and Serra-Capizzano, S.}
\newblock On the spectrum and asymptotic conditioning of a class of positive
  definite {T}oeplitz matrix-sequences, with application to
  fractional-differential approximations.
\newblock {\em arXiv:2112.02685\/} (2021).

\bibitem{BoGr98}
{\sc B\"ottcher, A., and Grudsky, {\relax S.M}.}
\newblock On the condition numbers of large semi-definite {T}oeplitz matrices.
\newblock {\em Linear Algebra Appl. 279}, 1/3 (1998), 285--301.

\bibitem{BoGr05}
{\sc B\"ottcher, A., and Grudsky, {\relax S.M}.}
\newblock {\em Spectral properties of banded {T}oeplitz matrices}.
\newblock Society for Industrial and Applied Mathematics (SIAM), Philadelphia,
  PA, 2005.

\bibitem{ChJi07}
{\sc Chan, R., and Jin, X.}
\newblock {\em An introduction to iterative {T}oeplitz solvers}.
\newblock Society for Industrial and Applied Mathematics (SIAM), 2007.

\bibitem{DiBSe99}
{\sc Di-Benedetto, F., and Serra-Capizzano, S.}
\newblock A unifying approach to abstract matrix algebra preconditioning.
\newblock {\em Numer. Math. 82}, 1 (1999), 57--90.

\bibitem{DoGa15}
{\sc Donatelli, M., Garoni, C., Manni, C., Serra-Capizzano, S., and Speleers,
  H.}
\newblock Robust and optimal multi-iterative techniques for {I}g{A} {G}alerkin
  linear systems.
\newblock {\em Comput. Methods Appl. Mech. Engrg. 284\/} (2015), 230--264.

\bibitem{DoGa17}
{\sc Donatelli, M., Garoni, C., Manni, C., Serra-Capizzano, S., and Speleers,
  H.}
\newblock Symbol-based multigrid methods for {G}alerkin {B}-spline isogeometric
  analysis.
\newblock {\em SIAM. J. Numer. Anal. 55}, 1 (2017), 31--62.

\bibitem{DoMa16}
{\sc Donatelli, M., Mazza, M., and Serra-Capizzano, S.}
\newblock Spectral analysis and structure preserving preconditioners for
  fractional diffusion equations.
\newblock {\em J. Comput. Phys. 307\/} (2016), 262--279.

\bibitem{DoMa18}
{\sc Donatelli, M., Mazza, M., and Serra-Capizzano, S.}
\newblock Spectral analysis and multigrid methods for finite volume
  approximations of space-fractional diffusion equations.
\newblock {\em SIAM J. Sci. Comput. 40}, 6 (2018), A4007--A4039.

\bibitem{ErRo06}
{\sc Ervin, V., and Roop, J.}
\newblock Variational formulation for the stationary fractional advection
  dispersion equation.
\newblock {\em Numer. Methods Partial Differ. Equ. 22}, 3 (2006), 558--576.

\bibitem{FiSe92}
{\sc Fiorentino, G., and Serra-Capizzano, S.}
\newblock Multigrid methods for {T}oeplitz matrices.
\newblock {\em Calcolo 28}, 3-4 (1992), 283--305.

\bibitem{GaSe17}
{\sc Garoni, C., and Serra-Capizzano, S.}
\newblock {\em Generalized {L}ocally {T}oeplitz sequences: {T}heory and
  applications. {V}ol. {I}}.
\newblock Springer, Cham, 2017.

\bibitem{HaWo85}
{\sc Hackbusch, W.}
\newblock {\em Multigrid methods and applications}, vol.~4 of {\em Springer
  Series in Computational Mathematics}.
\newblock Springer-Verlag, Berlin, 1985.

\bibitem{HeMo14}
{\sc Hejazi, H., Moroney, T., and Liu, F.}
\newblock Stability and convergence of a finite volume method for the space
  fractional advection-dispersion equation.
\newblock {\em J. Comput. Appl. Math. 255\/} (2014), 684--697.

\bibitem{LiWa18}
{\sc Lin, Z., and Wang, D.}
\newblock A finite element formulation preserving symmetric and banded
  diffusion stiffness matrix characteristics for fractional differential
  equations.
\newblock {\em Comput. Mech. 62}, 2 (2018), 185--211.

\bibitem{MaKa18}
{\sc Mao, Z., and Karniadakis, G.}
\newblock A spectral method (of exponential convergence) for singular solutions
  of the diffusion equation with general two-sided fractional derivative.
\newblock {\em SIAM J. Numer. Anal. 56}, 1 (2018), 24--49.

\bibitem{MaSe21}
{\sc Mazza, M., Serra-Capizzano, S., and Usman, M.}
\newblock Symbol-based preconditioning for {R}iesz distributed-order
  space-fractional diffusion equations.
\newblock {\em Electr. Trans. Num. Anal. 54\/} (2021), 499--513.

\bibitem{MeTa04}
{\sc Meerschaert, M., and Tadjeran, C.}
\newblock Finite difference approximations for fractional advection-dispersion
  flow equations.
\newblock {\em J. Comput. Appl. Math. 172}, 1 (2004), 65--77.

\bibitem{Ng04}
{\sc Ng, M.}
\newblock {\em Iterative methods for {T}oeplitz systems}.
\newblock Oxford University Press, 2004.

\bibitem{NoSe04}
{\sc Noutsos, D., Serra-Capizzano, S., and Vassalos, P.}
\newblock Matrix algebra preconditioners for multilevel {T}oeplitz systems do
  not insure optimal convergence rate.
\newblock {\em Theoret. Comput. Sci. 315}, 2-3 (2004), 557--579.

\bibitem{PeTa11}
{\sc Pedas, A., and Tamme, E.}
\newblock On the convergence of spline collocation methods for solving
  fractional differential equations.
\newblock {\em J. Comput. Appl. Math. 235}, 12 (2011), 3502--3514.

\bibitem{Se95}
{\sc Serra-Capizzano, S.}
\newblock New {PCG} based algorithms for the solution of {H}ermitian {T}oeplitz
  systems.
\newblock {\em Calcolo 32\/} (1995), 53--176.

\bibitem{Se98}
{\sc Serra-Capizzano, S.}
\newblock On the extreme eigenvalues of {H}ermitian {\rm (}block{\rm )}
  {T}oeplitz matrices.
\newblock {\em Linear Algebra Appl. 270\/} (1998), 109--129.

\bibitem{Se99a}
{\sc Serra-Capizzano, S.}
\newblock A {K}orovkin-type theory for finite {T}oeplitz operators via matrix
  algebras.
\newblock {\em Numer. Math. 82}, 1 (1999), 117--142.

\bibitem{Se99b}
{\sc Serra-Capizzano, S.}
\newblock The rate of convergence of {T}oeplitz based {PCG} methods for second
  order nonlinear boundary value problems.
\newblock {\em Numer. Math. 81}, 3 (1999), 461--495.

\bibitem{Se99c}
{\sc Serra-Capizzano, S.}
\newblock Toeplitz preconditioners constructed from linear approximation
  processes.
\newblock {\em SIAM J. Matrix Anal. Appl. 20}, 2 (1999), 446--465.

\bibitem{Se02a}
{\sc Serra-Capizzano, S.}
\newblock Convergence analysis of two-grid methods for elliptic {T}oeplitz and
  {PDE}s matrix-sequences.
\newblock {\em Numer. Math. 92}, 3 (2002), 433--465.

\bibitem{Se02b}
{\sc Serra-Capizzano, S.}
\newblock Matrix algebra preconditioners for multilevel {T}oeplitz matrices are
  not superlinear. {S}pecial issue on structured and infinite systems of linear
  equations.
\newblock {\em Linear Algebra Appl. 343/344\/} (2002), 303--319.

\bibitem{SeTy99}
{\sc Serra-Capizzano, S., and Tyrtyshnikov, E.}
\newblock Any circulant-like preconditioner for multilevel matrices is not
  superlinear.
\newblock {\em SIAM J. Matrix Anal. Appl. 21}, 2 (1999), 431--39.

\bibitem{TiZh15}
{\sc Tian, W., Zhou, H., and Deng, W.}
\newblock A class of second order difference approximations for solving space
  fractional diffusion equations.
\newblock {\em Math. Comput. 84}, 294 (2015), 1703--1727.

\bibitem{WaDu13}
{\sc Wang, H., and Du, N.}
\newblock A superfast-preconditioned iterative method for steady-state
  space-fractional diffusion equations.
\newblock {\em J. Comput. Phys. 240\/} (2013), 49--57.

\bibitem{XuDa20}
{\sc Xu, K., and Darve, E.}
\newblock Isogeometric collocation method for the fractional {L}aplacian in the
  2{D} bounded domain.
\newblock {\em Computer Methods in Applied Mechanics and Engineering 364\/}
  (2020), 112936.

\bibitem{ZeMa17}
{\sc Zeng, F., Mao, Z., and Karniadakis, G.}
\newblock A generalized spectral collocation method with tunable accuracy for
  fractional differential equations with end-point singularities.
\newblock {\em SIAM J. Sci. Comput. 39}, 1 (2017), A360--A383.

\end{thebibliography}
\end{document}